\documentclass[11pt,leqno]{amsart}
\usepackage{amsmath,amsthm,amssymb}
\usepackage{tabularx}
\usepackage{comment}
\usepackage{url} 

\newcommand{\GF}{{\mathbb F}}
\newcommand{\FF}{{\mathbb F}}

\newcommand{\R}{{\mathbb R}}
\newcommand{\RR}{{\mathbb R}}

\newcommand{\NN}{{\mathbb N}}

\newcommand{\Aut}{{\rm Aut}}

\newcommand{\wt}{{\rm wt}}
\newcommand{\supp}{{\rm supp}}

\DeclareMathOperator{\Harm}{Harm}

\usepackage{color} 
\usepackage{ulem}

\newtheorem{thm}{Theorem}[section]
\newtheorem{lem}[thm]{Lemma}
\newtheorem{cor}[thm]{Corollary}

\theoremstyle{definition}
\newtheorem{df}[thm]{Definition}
\newtheorem{rem}[thm]{Remark}
\newtheorem{ex}[thm]{Example}

\numberwithin{equation}{section}
\title[A note on $t$-designs in isodual codes]
{A note on $t$-designs in isodual codes}

\author[Awada]{Madoka Awada}
\address{School of Fundamental Science and Engineering, 
Waseda University, 
Tokyo 169--8555, Japan
}
\email{madoka-awada@fuji.waseda.jp} 
\author[Miezaki]{Tsuyoshi Miezaki*}
\thanks{*Corresponding author}
\address{		Faculty of Science and Engineering, 
		Waseda University, 
		Tokyo 169--8555, Japan
}
\email{miezaki@waseda.jp} 
\author[Munemasa]{Akihiro Munemasa}
\address{Graduate School of Information Sciences, Tohoku University,
Tohoku 980-8579, Japan}
\email{munemasa@math.is.tohoku.ac.jp}

\author[Nakasora]{Hiroyuki Nakasora}
\address{Faculty of Computer Science and Systems Engineering, 
Okayama Prefectural University, 
Okayama 719-1197, Japan}
\email{nakasora@cse.oka-pu.ac.jp}



\begin{document}
\begin{abstract}
In the present paper, 
%we show that for an even integer $t$, 
%if there exists a $t$-design in an isodual code, 
%then it is a $(t+1)$-design. 
%As a corollary, 
we construct 3-designs using 
extended binary quadratic residue codes and their dual codes. 
\end{abstract}
\maketitle

%%%%%%%%%%%%%%%%%%%%%%%%%%%%%%%%%%%%%%%%%%%%%%%%%%%%%%%%%%%%%%%%%%%%%%%%%%%%%%

\section{Introduction}

Let $C$ be a code and 
$C_\ell:=\{c\in C\mid \wt(c)=\ell\}$. 
In this paper, we call $C_\ell$ a shell of the code $C$ 
whenever it is non-empty.
Let $p\equiv 1\pmod{8}$, 
$Q_p$ be the binary quadratic residue code of length $p$, and 
$\widetilde{Q}_{p+1}$ be the extended binary quadratic residue code 
of length $p+1$. 
Shells of $\widetilde{Q}_{p+1}$ are known to 
support combinatorial $2$-designs by the 
transitivity argument (see Example~\ref{ex:AutQR}). 
%Assmus--Mattson theorem (see Theorem \ref{thm:assmus-mattson}).
More precisely, 
%let $C$ be the extended quadratic residue code. 
%Then 
the set $\mathcal{B}((\widetilde{Q}_{p+1})_{\ell}):=
\{\supp(x)\mid x\in (\widetilde{Q}_{p+1})_\ell\}$
forms the set of blocks of a combinatorial $2$-design. 

%Let $p\equiv 1\pmod{4}$, 
%$Q_p$ be a binary quadratic residue code of length $p$, and 
%$\widetilde{Q}_{p+1}$ be the extended binary quadratic residue code 
%of length $p+1$. 
In this paper, by computing Jacobi polynomials and 
harmonic weight enumerators of $\widetilde{Q}_{p+1}$, 
we show that $(\widetilde{Q}_{p+1})_\ell \cup (\widetilde{Q}_{p+1}^\perp)_\ell$ is a 3-design 
whenever $(\widetilde{Q}_{p+1})_\ell \cup {(\widetilde{Q}_{p+1}^\perp)}_\ell$ is non-empty. 
We use $\sqcup$ to denote disjoint union and 
a code $C$ is isodual if $C$ and $C^\perp$ are equivalent. 

\begin{thm}\label{thm:main}
Let $C$ be an isodual binary code of length $n$, 
$X:=\{1,\ldots,n\}$, and 
$G=\Aut(C)$. 
Let $\sigma \in S_n$ such that $C^\perp=C^\sigma$. 
Then $G$ acts on $\binom{X}{t}$ and 
we assume that 
$G$ has two orbits: 
%and 
%we denote by $T_1$ and $T_2$ their representatives: 
\[
\binom{X}{t}=GT_1\sqcup GT_2 
\]
such that $(GT_1)^\sigma=GT_2$. 
Then the following statements hold: 
\begin{enumerate}
\item [(1)]
%For all $T\in \binom{X}{t}$, 
$J_{C,T}+J_{C^\perp,T}$ is independent 
of the choice of $T$ with $|T|={t}$. 

\item [(2)]
Let $f$ be a harmonic function of a degree $t$, 
which is an invariant of $\Aut(C)$. 
Then we have 
\[
w_{C,f}+w_{C^\perp,f}=0. 
\]

\end{enumerate}
\end{thm}

Applying Theorem~\ref{thm:main}, 
we have the following corollary: 
%we give a 3-design in the extended quadratic residue code and its dual. 
%\begin{cor}\label{cor:main}
%Let $C$ be an isodual code and $t$ be an even integer. 
%We assume that $C_\ell$ be a $t$-design. 
%Then $C_\ell$ is a $(t+1)$-design. 
%\end{cor}

\begin{cor}\label{cor:main}
Let $p\equiv 1\pmod{8}$ and
$\widetilde{Q}_{p+1}$ be the extended binary quadratic residue code of length $p+1$. 
Then for $\ell\in \NN$, $(\widetilde{Q}_{p+1})_\ell \cup {(\widetilde{Q}_{p+1}^\perp)}_\ell$ is a $3$-design 
whenever $(\widetilde{Q}_{p+1})_\ell \cup {(\widetilde{Q}_{p+1}^\perp)}_\ell$ is non-empty. 
\end{cor}

\begin{rem}
We note that 
$\widetilde{Q}_{p+1}\cap \widetilde{Q}_{p+1}^\perp=\{0,\mathbf{1}\}$ 
(see \cite[Chapter 6, Theorem 70]{Pless}). 
\end{rem}

This paper is organized as follows. 
In Section~\ref{sec:pre}, 
we give definitions and some basic properties of the 
codes, 
combinatorial $t$-designs, 
Jacobi polynomials, and harmonic weight enumerators 
used in this paper.
In Section~\ref{sec:main}, 
we give proofs of Theorem~\ref{thm:main} and 
Corollary~\ref{cor:main}.
%In Section~\ref{sec:96}, 
%we give a proof of Theorem~\ref{thm:ub}, 
%give the known examples of 
%binary doubly even self-dual codes of length $96$ 
%with minimum weight $16$ and 
%investigate their designs. 
Finally, in Section~\ref{sec:rem} 
we give concluding remarks about the 3-design, which was 
recently obtained by Bonnecaze and Sol\'e~\cite{BS}.  

All computer calculations presented in this paper were done using 
{\sc Magma}~\cite{Magma} and {\sc Mathematica}~\cite{Mathematica}. 

\section{Preliminaries}\label{sec:pre}

\subsection{Codes and combinatorial $t$-designs}

A binary linear code $C$ of length $n$ is a linear subspace of $\FF_{2}^{n}$. 
%{
An inner product $({x},{y})$ on $\FF_2^n$ is given 
by
\[
(x,y)=\sum_{i=1}^nx_iy_i,
\]
where $x,y\in \FF_2^n$ with $x=(x_1,x_2,\ldots, x_n)$ and 
$y=(y_1,y_2,\ldots, y_n)$. 
The dual of a linear code $C$ is defined as follows: 
\[
C^{\perp}=\{{y}\in \FF_{2}^{n}\mid ({x},{y}) =0 \text{ for all }{x}\in C\}. 
\]
%A linear code $C$ is called self-dual 
%if $C=C^{\perp}$ for $q=2$ and $q=3$ and 
%if $C=C^{\perp,H}$ for $q=4$. 
For $x \in\FF_2^n$,
the weight $\wt(x)$ is the number of its nonzero components. 
%In this paper, we consider the following self-dual codes~\cite{CS}: 
%\begin{tabbing}
%Doubly even: A code is defined over $\FF_{2}^{n}$ with all weights divisible by% $4$,\\
%Ternary: A code is defined over $\FF_{3}^{n}$ with all weights divisible by $3$%,\\
%Quaternary: A code is defined over $\FF_{4}^{n}$ with all weights divisible by %$2$. 
%\end{tabbing}

Let $C$ be a cyclic code of length $n$. 
Namely, if $(c_1,c_2,\ldots,c_n)\in C$ then 
$(c_{n},c_1,\ldots,c_{n-1})\in C$. 
Then 
$C$ corresponds to an ideal $(g)$ of 
\[
\FF_2[x]/(x^n-1). 
\]
We call $g$ a generator polynomial of $C$. 
For $p\equiv \pm 1\pmod{8}$, 
the binary quadratic residue code $Q_p$ is 
a cyclic code of length $p$, 
which is generated by 
\[
\prod_{\ell \in (\FF_p^\ast)^2}(x-\alpha^\ell), 
\]
where 
$(\FF_p^\ast)^2=\{\ell^2\mid \ell\in \FF_p\}$ and $\alpha$ is a primitive root of order $p$. 
For details of the binary quadratic residue codes, see \cite{MS}. 

Let $C$ be a code of length $n$. 
Then the symmetric group $S_n$ acts on the $n$ coordinates of $C$. 
The automorphism group $\Aut(C)$ of $C$ is the subgroup of $S_n$ such that 
\[
\Aut(C):=\{\sigma\in S_n\mid C^\sigma=C\}, 
\]
where 
\[
C^\sigma
:=
\{(c_{\sigma(1)},\ldots,c_{\sigma(n)})\mid (c_1,\ldots,c_n)\in C\}. 
\]
\begin{ex}\label{ex:AutQR}
Let $\widetilde{Q}_{p+1}$ be the extended binary quadratic 
residue code of length $p+1$: 
\[
\widetilde{Q}_{p+1}
=\{(c_1,\ldots,c_p,c_{p+1})\in \FF_2^{p+1}\mid (c_1,\ldots,c_p)\in Q_p,
\sum_{i=1}^{p+1} c_i=0\}. 
\]
For $p\equiv 1\pmod{8}$, 
the automorphism group of 
%the extended binary quadratic residue code of length $p+1$ 
$\widetilde{Q}_{p+1}$ is $PSL_2(p)$ 
\cite[Chapter 6, Theorem 6.6.27]{HP}
(see also \cite{{assmus-mattson},{Blahut},{Huffman}}). 
%We assume that the coordinates of 
%${Q}_{p}$ 
%are labeled by $\{\infty,0,1,\ldots,p-1\}$ and 
%that of 
We label the new coordinate $\infty$ and 
assume that the coordinates of 
$\widetilde{Q}_{p+1}$ 
are labeled by $\{0,1,\ldots,p-1,\infty\}$. 
%Then $PSL_2(p)$ is generated by the three elements: 
%\begin{itemize}
%\item
%$\sigma:i\mapsto i+1 \pmod{p}$, 
%\item
%$\tau_a:i\mapsto ai \pmod{p}\ a\in (\FF_p^\ast)^2$, 
%\item
%$\rho:i\mapsto -1/i \pmod{p}\ i\neq 0$. 
%\end{itemize}
We identify $\{0,1,\ldots,p-1,\infty\}$ with $PG(1,p)$. 
Then the action of $PGL_2(p)$ on $PG(1,p)$ is $3$-transitive 
(see \cite[Propositions 4.6 and 4.8]{BJL}) and 
the action of $PSL_2(p)$ on $PG(1,p)$ 
%is $2$-transitive 
%but 
is not 3-homogeneous 
(see \cite{BR}). 
%(see \cite{{BR},{[Chapter 2, Remark 2.4]{Ebeling}}}). 
Indeed, 
%we have two representatives of orbits as follows. 
let $a$ be a generator of $\FF_p^\ast$. 
Then 
%For $p\equiv 1\pmod{8}$, $PSL_2(p)$ that acts on $\{0,1,\ldots,p-1,\infty\}$ 
%has two orbits: 
\[
\binom{X}{3}
=PSL_2(p)\{0,1,\infty\}\sqcup PSL_2(p)\{0,a,\infty\}. 
\]
%As the action of $PSL_2(p)$ on $PG(1,p)$ is $2$-transitive for all $p$, 
%there exists $\bar{A} \in PSL_2(p)$ that maps any subset 
%$\{\alpha,\beta,\gamma\}$ to $\{\infty,0,a^k\}$.
%Let $B$, $C$, and $D$ be the followings: 
%\[
%B:=\begin{pmatrix} a&0\\0&a^{-1} \end{pmatrix}, 
%C:=\begin{pmatrix} 1&0\\-1&1 \end{pmatrix}, 
%D:=\begin{pmatrix} 1&-a\\0&1 \end{pmatrix}.
%\]
%Then, $\bar{B} \in PSL_2(p)$ maps $\{\infty,0,a^k\}$ to $\{\infty,0,1\}$, 
%$\{\infty,0,-1\}$, $\{\infty,0,a\}$, or $\{\infty,0,-a\}$, $\bar{C} \in PSL_2(p%)$ maps 
%$\{\infty,0,1\}$ to $\{\infty,0,-1\}$, and $\bar{D} \in PSL_2(p)$ maps $\{\inft%y,0,a\}$ 
%to $\{\infty,0,-a\}$. Also, we confirm that there exist no elements of $PSL_2(p%)$ 
%that map $\{\infty,0,-1\}$ to $\{\infty,0,-a\}$ or $\{\infty,0,-a\}$ to $\{\inf%ty,0,-1\}$. 
\end{ex}

A combinatorial $t$-$(n,\ell,\lambda)$ design  (or $t$-design for short)
is a pair 
$\mathcal{D}=(\Omega,\mathcal{B})$, where $\Omega$ is a set of points of 
cardinality $n$, and $\mathcal{B}$ is a collection of $\ell$-element subsets
of $\Omega$ called blocks, with the property that any $t$ points are 
contained in precisely $\lambda$ blocks.

%A combinatorial $t$-design 
%is a pair 
%$\mathcal{D}=(\Omega,\mathcal{B})$, where $\Omega$ is a set of points of 
%cardinality $v$, and $\mathcal{B}$ is a collection of $k$-element subsets
%of $\Omega$ called blocks, with the property that any $t$ points are 
%contained in precisely $\lambda$ blocks.

The support of a vector 
${x}:=(x_{1}, \dots, x_{n})$, 
$x_{i} \in \GF_{2}$ is 
the set of indices of its nonzero coordinates: 
${\rm supp} ({x}) = \{ i \mid x_{i} \neq 0 \}$\index{$supp (x)$}.
Let $\Omega:=\{1,\ldots,n\}$ and 
$\mathcal{B}(C_\ell):=\{\supp({x})\mid {x}\in C_\ell\}$. 
Then for a code $C$ of length $n$, 
we say that $C_\ell$ is a combinatorial $t$-design if 
$(\Omega,\mathcal{B}(C_\ell))$ is a combinatorial $t$-design. 

%The following theorem from 
%Assmus and Mattson \cite{assmus-mattson} is one of the 
%most important theorems in coding theory and design theory.
%\begin{thm}[\cite{assmus-mattson}] \label{thm:assmus-mattson}
%Let $C$ be a linear code of 
%length $n$ over $\FF_q$ with minimum weight $d$. 
%Let $C^\perp$ denote the dual code of $C$, with 
%minimum weight $d^\perp$. 
%Suppose that an integer $t$ $($$1 \leq t \leq n$$)$ is 
%such that there are at most $d-t$ weights of $C^\perp$ 
%in $\{1, 2,\ldots , n - t\}$, 
%or such that there are at most $d^\perp - t$ weights of $C$ 
%in $\{1, 2, \ldots , n-t\}$. 
%Then the supports of the words of any fixed weight 
%in $C$ form a $t$-design $($with possibly repeated blocks$)$.
%\end{thm}

The following lemma is easily seen. 

\begin{lem}[{{\cite[Page 3, Proposition 1.4]{CL}}}]\label{lem: divisible}
Let $\lambda(S)$ be the number of blocks containing a given set $S$
of $s$ points in a combinatorial $t$-$(n,\ell,\lambda)$ design, where $0\leq s\leq t$. Then
\[
\lambda(S)\binom{\ell-s}{t-s}
=
\lambda\binom{n-s}{t-s}. 
\]
%In particular, the number of blocks is
%\[\frac{v(v-1)\cdots(v-t+1)}{k(k-1)\cdots(k-t+1)}\lambda.\]
\end{lem}

\subsection{Jacobi polynomials}

Let $C$ be a binary code of length $n$ and $T\subset [n]:=\{1,\ldots,n\}$. 
Then the Jacobi polynomial of $C$ with $T$ is defined as follows \cite{Ozeki}:
\[
J_{C,T} (w,z,x,y) :=\sum_{c\in C}w^{m_0(c)} z^{m_1(c)}x^{n_0(c)}y^{n_1(c)}, 
\]
where for $c=(c_1,\ldots,c_n)$, 
\begin{align*}
m_i(c)&=|\{j\in T\mid c_j=i \}|,\\
n_i(c)&=|\{j\in [n]\setminus T\mid c_j=i \}|.
\end{align*}
%$m_i(c)$ is the Hamming composition of $c$ on $T$ 
%and $n_i(c)$ is the Hamming composition of $c$ on $[n]\setminus T$. 
%The following is a generalization of the classical MacWilliams identity: 
%\begin{thm}[\cite{Ozeki}]\label{thm:Mac-Jacobi}
%Let $C$ be a binary code of length $n$ and $T\subset [n]$. 
%Then we have 
%\[
%J_{C^\perp,T}(w,z,x,y) =\frac{1}{|C|}J_{C,T}(w + z,w - z,x + y,x - y).
%\]
%\end{thm}
\begin{rem}\label{rem:hom}

It is easy to see that 
$C_\ell$ is a combinatorial $t$-design 
if and only if 
the coefficient of $z^{t}x^{n-\ell}y^{\ell-t}$ 
in $J_{C,T}$ is independent of the choice of $T$ with $|T|=t$. 
%A code $C$ is $t$-homogeneous if 
%for all $\ell\in \NN$, $C_\ell$ is a combinatorial $t$-design 
%if $C_\ell\neq \emptyset$. 
%It is easy to see that 
%$C$ is $t$-homogeneous if and only if 
%$J_{C,T}$ is independent 
%of the choice of $T$ with $|T|={t}$. 
%%of the choices of $T\in \binom{X}{t}$. 
Moreover, for all $\ell\in \NN$, $C_\ell\cup C_\ell^\perp$ is a combinatorial $t$-design 
if and only if 
$J_{C,T}+J_{C^\perp,T}$ is independent 
of the choice of $T$ with $|T|={t}$. 
%of the choices of $T\in \binom{X}{t}$. 

\end{rem}

\subsection{Harmonic weight enumerators}\label{sec:Har}

%In this section, we extend the method of 
%harmonic weight enumerators 
%which were used by Bachoc~\cite{Bachoc} and 
%Bannai et al.~\cite{Bannai-Koike-Shinohara-Tagami}.
%For the readers convenience we quote from~\cite{Bachoc,Delsarte}
%the definitions and properties of discrete harmonic functions 
%(for more information the reader is referred to~\cite{Bachoc,Delsarte}). 

In this subsection, we review the concept of 
harmonic weight enumerators.

Let $\Omega=\{1, 2,\ldots,n\}$ be a finite set (which will be the set of coordinates of the code) and 
let $X$ be the set of its subsets,  while for all $k= 0,1,\dots, n$, 
$X_{k}$ is the set of its $k$-subsets.
We denote by $\R X$ and $\R X_k$ the 
real vector spaces spanned by the elements of $X$
and $X_{k}$, respectively.
An element of $\R X_k$ is denoted by
$$f=\sum_{z\in X_k}f(z)z$$
and is identified with the real-valued function on $X_{k}$ given by 
$z \mapsto f(z)$. 

An element $f\in \R X_k$ can be extended to an element $\widetilde{f}\in \R X$ by setting
$$\widetilde{f}(u)=\sum_{z\in X_k, z\subset u}f(z)$$
for all $u \in X$. If an element $g \in \R X$ is equal to $\widetilde{f}$ 
for some $f \in \R X_{k}$, then we say that $g$ has degree $k$. 
The differentiation $\gamma$ is the operator on $\RR X$ defined by linearity from 
$$\gamma(z) =\sum_{y\in X_{k-1},y\subset z}y$$
for all $z\in X_k$ and for all $k=0,1, \ldots, n$, and $\Harm_{k}$ is the kernel of $\gamma$:
$$\Harm_k =\ker(\gamma|_{\R X_k}).$$

\begin{thm}[{{\cite[Theorem 7]{Delsarte}}}]\label{thm:design}
A set $\mathcal{B} \subset X_{m}$ $($where $m \leq n$$)$ is 
the set of blocks of a $t$-design 
if and only if $\sum_{b\in \mathcal{B}}\widetilde{f}(b)=0$ 
for all $f\in \Harm_k$, $1\leq k\leq t$. 

\end{thm}
The symmetric group $S_n$ acts on $\Omega$ and 
the automorphism group $\Aut(\mathcal{B})$ of $\mathcal{B}$ is 
the subgroup of $S_n$ such that 
\[
\Aut(\mathcal{B}):=\{\sigma\in S_n\mid \mathcal{B}^\sigma=\mathcal{B}\}, 
\]
where 
\[
\mathcal{B}^\sigma:=\{ \{\sigma(b_1),\ldots,\sigma(b_m)\}\mid 
\{b_1,\ldots,b_m\}\in \mathcal{B}\}. 
\]
Let $G$ be a subgroup of $\Aut(\mathcal{B})$. 
Then $G$ acts on $\Harm_k$ through the above action and 
we denote by $\Harm_k^{G}$ the set of the invariants of $G$: 
\[
\Harm_k^{G}=\{f\in {\rm Harm}_k\mid f^\sigma=f, \forall \sigma\in G\}, 
\]
where $f^\sigma$ is defined by linearity from 
\[
\{{i_1},\ldots, {i_k}\}^\sigma=\{{\sigma(i_1)},\ldots, {\sigma(i_k)}\}. 
\]
Theorem~\ref{thm:design} can be reinterpreted in terms of $\Harm_k^{G}$. 
\begin{thm}\label{thm:design-aut}
A set 
$\mathcal{B} \subset X_{m}$ $($where $m \leq n$$)$ is 
the set of blocks of a $t$-design 
if and only if $\sum_{b\in \mathcal{B}}\widetilde{f}(b)=0$ 
for all $f\in \Harm_k^{G}$, $1\leq k\leq t$. 

\end{thm}
\begin{proof}

We assume that 
$\mathcal{B} \subset X_{m}$ is a $t$-design. 
Let $G<\Aut(\mathcal{B})$ and 
$f\in \Harm_k^{G}$ ($1\leq k\leq t$). 
Then by Theorem~\ref{thm:design}, 
$\sum_{b\in \mathcal{B}}\widetilde{f}(b)=0$. 

We assume that 
for all $f\in \Harm_k^{G}$ ($1\leq k\leq t$), 
$\sum_{b\in \mathcal{B}}\widetilde{f}(b)=0$. 
Let 
\[
\mathcal{B}=Gx_1\sqcup \cdots \sqcup G{x_s}, 
\]
and $f\in \Harm_k$ ($1\leq k\leq t$). 
Then 
\begin{align*}
\sum_{b\in \mathcal{B}}\widetilde{f}(b)
&=
\sum_{x\in Gx_1}\widetilde{f}(x)+ \cdots + \sum_{x\in G{x_s}}\widetilde{f}(x)\\
&=
\frac{1}{|G_{x_1}|}\sum_{\sigma\in G}\widetilde{f}^{\sigma}(x_1)+ \cdots 
+ \frac{1}{|G_{x_s}|}\sum_{\sigma\in G}\widetilde{f}^{\sigma}(x_s)=0, 
\end{align*}
since for each $i\in \{1,\ldots,s\}$, 
\[
\frac{1}{|G_{x_i}|}\sum_{\sigma\in G}\widetilde{f}^\sigma
\]
is an invariant polynomial. 
\end{proof}
%for all $f\in \Harm_k$, $1\leq k\leq t$, 
%$\sum_{b\in \mathcal{B}}\widetilde{f}(b)=0$ 
%if and only if 
%for all $f\in \Harm_k^{\Aut(B)}$, $1\leq k\leq t$,  
%$\sum_{b\in \mathcal{B}}\widetilde{f}(b)=0$. 

In \cite{Bachoc}, the harmonic weight enumerator associated with a binary linear code $C$ was defined as follows. 
\begin{df}
Let $C$ be a binary code of length $n$ and let $f\in\Harm_{k}$. 
The harmonic weight enumerator associated with $C$ and $f$ is

$$w_{C,f}(x,y)=\sum_{{c}\in C}\widetilde{f}({\supp (c)})x^{n-\wt({c})}y^{\wt({c})}.$$
\end{df}

%Bachoc proved the following MacWilliams-type equality: 
%\begin{thm}[\cite{Bachoc}] \label{thm: Bachoc iden.} 
%Let $w_{C,f}(x,y)$ be 
%the harmonic weight enumerator associated with the code $C$ 
%and the harmonic function $f$ of degree $k$. Then 
%$$w_{C,f}(x,y)= (xy)^{k} Z_{C,f}(x,y),$$
%where $Z_{C,f}$ is a homogeneous polynomial of degree $n-2k$, and satisfies
%$$Z_{C^{\bot},f}(x,y)= (-1)^{k} \frac{2^{n/2}}{|C|} Z_{C,f} \left( \frac{x+y}{\%sqrt{2}}, \frac{x-y}{\sqrt{2}} \right).$$
%\end{thm}

\begin{rem}\label{rem:hom-2}
It follows from Theorems~\ref{thm:design} and \ref{thm:design-aut} that 
$C_\ell$ is a combinatorial $t$-design if and only if 
the coefficient of $x^{n-\ell}y^\ell$ in $w_{C,f}(x,y)$ vanishes 
for all $f\in \Harm_k^{\Aut(C)}\ (1\leq k\leq t)$. 
Moreover, for all $\ell\in \NN$, $C_\ell\cup C_\ell^\perp$ is a combinatorial $t$-design 
if and only if 
$w_{C,f}+w_{C^\perp,f}=0$ for all $f\in \Harm_k^{\Aut(C)}\ (1\leq k\leq t)$. 
%of the choices of $T\in \binom{X}{t}$. 
\end{rem}

\section{Proof of Theorem~\ref{thm:main}}\label{sec:main}

In this section, 
we give a proof of Theorem~\ref{thm:main}. 
\begin{proof}[Proof of Theorem~\ref{thm:main}]

\begin{enumerate}
\item [(1)]

We recall that 
$C^{\sigma}=C^\perp$ and $(GT_1)^\sigma=GT_2$. 
We note that 
for all $T\in GT_i\ (i\in\{1,2\})$, $J_{C,T}=J_{C,T_i}$. 
Then for any $T\in \binom{X}{t}$, 
%Let $\sigma\in S_n$ such that $C^{\sigma}=C^\perp$. 
%Then we may assume $T_1^\sigma=T_2$ and for any $T\in \binom{X}{t}$, 
\begin{align*}
J_{C,T}+J_{C^\perp,T}&=J_{C,T}+J_{C^\sigma,T}\\
&=J_{C,T}+J_{C,T^{\sigma^{-1}}}\\
&=J_{C,T_1}+J_{C,T_2}. 
\end{align*}
Hence, $J_{C,T}+J_{C^\perp,T}$ is  independent 
of the choice of $T$ with $|T|={t}$. 
%of the choices of $T\in \binom{X}{t}$. 

\item [(2)]
Let $G:=\Aut(C)$. 
%For $T=\{i_1,\ldots,i_t\}\in \binom{X}{t}$, 
%$x_T:=x_{i_1}\cdots x_{i_t}$. 
For $f\in {\rm Harm}_{t}^{G}$, 
$f$ is written as a linear combination of 
$R({T_1})$ and $R({T_2})$, 
where $R$ is the Reynolds operator: for $T\in \binom{X}{t}$, 
\[
R({T})=\frac{1}{|G|}\sum_{\sigma\in G}{T}^{\sigma}. 
\]
Based on the above assumption, $|GT_1|=|GT_2|$. 
Then there exists a constant $c\in \RR$ such that 
\[
{\rm Harm}_{t}^{G}=\langle c(R({T_1})-R({T_2}))\rangle. 
\]
Let $f:=R({T_1})-R({T_2})$. 
Then 
\begin{align*}
w_{C,f}+w_{C^\perp,f}&=w_{C,f}+w_{C^\sigma,f}\\
&=w_{C,f}+w_{C,f^{\sigma^{-1}}}\\
&=w_{C,f}+w_{C,-f}=0.%\\
%&=w_{C,f}-w_{C,f}=0. 
\end{align*}
\end{enumerate}
This completes the proof of Theorem~\ref{thm:main}. 
\end{proof}
%Before proving the next corollary, 
%we quote a theorem. 
%\begin{thm}[\cite{planetmath}]\label{thm:orbit}
%Let $H$ be a normal subgroup of $G$, and assume 
%$G$ acts transitively on the finite set $A$. 
%Let $O_1,\ldots,O_r$ be the orbits of $H$ on $A$. 
%Then we have the following:
%\begin{enumerate}
%\item [(1)]
%$G$ permutes the $O_i$ transitively 
%$($i.e., for each $g\in G, 1\leq j\leq r$, 
%there is $1\leq k\leq r$ such that 
%$gO_j=O_k$, 
%and for each 
%$1\leq j,k\leq r$, 
%there is $g\in G$ such that 
%$gO_j=O_k$$)$, and the 
%$O_i$ all have the same cardinality. 
%\item [(2)]
%If $a\in O_i$, then 
%$|O_i|=|H:H\cap G_a|$ and 
%$r=|G:HG_a|$. 
%\end{enumerate}
%\end{thm}
\begin{proof}[Proof of Corollary~\ref{cor:main}]

Let $p\equiv 1\pmod{8}$ and 
$\widetilde{Q}_{p+1}$ be the extended binary quadratic residue code of length $p+1$. 
Then by Example~\ref{ex:AutQR}, 
\[
G:=\Aut(\widetilde{Q}_{p+1})\simeq PSL_2(p). 
\]
Let $X:=\{0,1,\ldots,p-1,\infty\}$. 
Then $G$ acts on $\binom{X}{3}$ and 
by Example~\ref{ex:AutQR}
%, $|PGL_2(p):PSL_2(p)|=2$ (see \cite[Chapter 2, Remark 2.4]{Ebeling}), 
and \cite[P.235, 8A.8.]{I}, 
we have two orbits: 
\[
\binom{X}{3}=GT_1\sqcup GT_2, 
\]
which satisfy the assumption of Theorem~\ref{thm:main}. 
By Theorem~\ref{thm:main}, 
%for all $T\in \binom{X}{t}$, 
\[
J_{\widetilde{Q}_{p+1},T}+J_{\widetilde{Q}_{p+1}^\perp,T}
\] 
 is independent 
of the choice of $T$ with $|T|={3}$, 
%of the choices of $T\in \binom{X}{3}$ 
and for $f\in {\rm Harm}_3^{G}$ 
we have 
\[
w_{\widetilde{Q}_{p+1},f}+w_{\widetilde{Q}_{p+1}^\perp,f}=0. 
\]
Since $(\widetilde{Q}_{p+1})_\ell$ and $(\widetilde{Q}_{p+1}^\perp)_\ell$ 
are 2-designs, 
for $f\in {\rm Harm}_k^{G}\ (1\leq k\leq 2)$, 
we have 
\[
w_{\widetilde{Q}_{p+1},f}+w_{\widetilde{Q}_{p+1}^\perp,f}=0. 
\]
Then by Remark~\ref{rem:hom} or \ref{rem:hom-2}, 
the proof is complete. 
%This completes the proof of Corollary \ref{cor:main}. 
\end{proof}

%\begin{cor}
%$RM(1,m)$ and $H_{2^m}$ are not $4$-homogeneous. 
%\end{cor}

%\begin{proof}
%By Theorem \ref{thm:Jacobi}, 
%for $T\subset [n]$ with $|T|=4$, 
%Jacobi polynomials are not uniquely determined. 
%\end{proof}

%We introduce the following notations: 
%\begin{align*}
%\delta(C)&:=\max\{t\in \mathbb{N}\mid \forall w, 
%C_{w} \mbox{ is a } t\mbox{-design}\},\\ 
%s(C)&:=\max\{t\in \mathbb{N}\mid \exists w \mbox{ s.t.~} 
%C_{w} \mbox{ is a } t\mbox{-design}\}.
%\end{align*}
%We remark that $\delta(C) \leq s(C)$ holds, and 
%we consider the possible occurrence of $\delta(C)<s(C)$.
%In \cite{MN-TEC}, the first and third authors gave 
%the first nontrivial examples of a code 
%that supports combinatorial $t$-designs for all weights 
%obtained from the Assmus--Mattson theorem and 
%that supports $t'$-designs for some weights with some $t'>t$ 
%(see also \cite{{BS},{Dillion-Schatz},{MMN},{mn-typeI}}). 

%After that, 
%an example was given by Bonnecaze and S\'ole \cite{BS}, who 
%found a 3-design in 
%the extended quadratic residue code of length 42. 
%They showed the existence of this design by 
%electronic calculation and 
%noted that 
%this design ``cannot be derived from the Assmus--Mattson Theorem, 
%and does not follow by the standard transitivity argument." 

\section{Quadratic residue code of length 42}\label{sec:rem}

Let $\widetilde{Q}_{42}$ be the extended binary quadratic residue code of length $42$. 
In \cite{BS}, Bonnecaze and Sol\'e 
found a 3-design $(\widetilde{Q}_{42})_{10}$ in $\widetilde{Q}_{42}$. 
%in the extended quadratic residue code of length 42.
They showed the existence of this design by 
electronic calculation and noted that 
this design ``cannot be derived from the Assmus--Mattson Theorem \cite{assmus-mattson}, 
and does not follow by the standard transitivity argument." 

In this section, 
we give Jacobi polynomials and harmonic weight enumerators of 
$\widetilde{Q}_{42}$. 
%the extended quadratic residue code of length 42. 
Then we present an alternative approach to 
the existence of the 3-design in 
$\widetilde{Q}_{42}$. 
%the extended quadratic residue code of length 42. 
Moreover, we show that 
$(\widetilde{Q}_{42})_{\ell}$ is not a $3$-design if $\ell\neq 10$ . 

%Then we present an alternative approach to 
%the existence of the 3-design in 
%the extended quadratic residue code of length 42. 
%Moreover, we show that 
%$C_{\ell}$ is not a $3$-design if $\ell\neq 10$ . 

Let $X:=\{0,1,\ldots,41,\infty\}$. For $T\in \binom{X}{3}$, 
the Jacobi polynomials of $\widetilde{Q}_{42}$ 
%the extended quadratic residue code of length 42 
are as follows. 
\begin{ex}\label{ex:42-Jac}
Let $\widetilde{Q}_{42}$ be the extended binary quadratic residue code of length 42, 
$X:=\{0,1,\ldots,41,\infty\}$, 
$T\in \binom{X}{3}$, and $G=\Aut(\widetilde{Q}_{42})$. 
Then by the fact that $\FF_{41}^\ast=\langle 6\rangle $ and 
Example~\ref{ex:AutQR}, 
\[
\binom{X}{3}=G\{0, 1,\infty\}\sqcup G\{0, 6,\infty\}
\]
and we have the following. 
%Then one of the following holds: 
\begin{enumerate}
\item [(1)]
%\begin{enumerate}
If $T_1\in G\{0, 1,\infty\}$, then
\begin{align*}
J_{\widetilde{Q}_{42},T_1}&(w,z,x,y)=w^3x^{39} + 744w^3x^{29}y^{10} + 3756w^3x^{27}y^{12} 
+ 14211w^3x^{25}y^{14} \\
&+ 35703w^3x^{23}y^{16}+ 60172w^3x^{21}y^{18} + 65436w^3x^{19}y^{20} + 48330w^3x^{17}y^{22} \\
&+ 24318w^3x^{15}y^{24} + 7668w^3x^{13}y^{26}+ 1584w^3x^{11}y^{28} + 203w^3x^9y^{30} \\
&+ 18w^3x^7y^{32} + 744w^2x^{30}y^9z + 4827w^2x^{28}y^{11}z+ 22977w^2x^{26}y^{13}z \\
&+71316w^2x^{24}y^{15}z + 147924w^2x^{22}y^{17}z + 195930w^2x^{20}y^{19}z \\
&+ 177630w^2x^{18}y^{21}z +109116w^2x^{16}y^{23}z + 42876w^2x^{14}y^{25}z \\
&+ 11043w^2x^{12}y^{27}z + 1833w^2x^{10}y^{29}z + 216w^2x^8y^{31}z \\
&+216wx^{31}y^8z^2 + 1833wx^{29}y^{10}z^2+ 11043wx^{27}y^{12}z^2 \\
&+ 42876wx^{25}y^{14}z^2 + 109116wx^{23}y^{16}z^2+177630wx^{21}y^{18}z^2 \\
&+ 195930wx^{19}y^{20}z^2 + 147924wx^{17}y^{22}z^2+ 71316wx^{15}y^{24}z^2 \\
&+22977wx^{13}y^{26}z^2 + 4827wx^{11}y^{28}z^2+ 744wx^9y^{30}z^2 \\
&+ 18x^{32}y^7z^3 + 203x^{30}y^9z^3 +1584x^{28}y^{11}z^3 \\
&+ 7668x^{26}y^{13}z^3 + 24318x^{24}y^{15}z^3 + 48330x^{22}y^{17}z^3 \\
&+ 65436x^{20}y^{19}z^3 +60172x^{18}y^{21}z^3 + 35703x^{16}y^{23}z^3 \\
&+ 14211x^{14}y^{25}z^3 + 3756x^{12}y^{27}z^3 + 744x^{10}y^{29}z^3 + y^{39}z^3.
\end{align*}

\item [(2)]
%If $u,v,w$ are linearly dependent, then
If $T_2\in G\{0, 6,\infty\}$, then
\begin{align*}
J_{\widetilde{Q}_{42},T_2}&(w,z,x,y)=w^3x^{39} + 744w^3x^{29}y^{10} + 3755w^3x^{27}y^{12} 
+ 14220w^3x^{25}y^{14} \\
&+ 35667w^3x^{23}y^{16} + 60256w^3x^{21}y^{18} + 65310w^3x^{19}y^{20} + 48456w^3x^{17}y^{22} \\
&+ 24234w^3x^{15}y^{24}+ 7704w^3x^{13}y^{26} + 1575w^3x^{11}y^{28} + 204w^3x^9y^{30} \\
&+ 18w^3x^7y^{32} + 744w^2x^{30}y^9z + 4830w^2x^{28}y^{11}z + 22950w^2x^{26}y^{13}z\\
 &+ 71424w^2x^{24}y^{15}z + 147672w^2x^{22}y^{17}z + 196308w^2x^{20}y^{19}z \\
&+ 177252w^2x^{18}y^{21}z +109368w^2x^{16}y^{23}z + 42768w^2x^{14}y^{25}z \\
&+ 11070w^2x^{12}y^{27}z + 1830w^2x^{10}y^{29}z+ 216w^2x^8y^{31}z \\
&+ 216wx^{31}y^8z^2 + 1830wx^{29}y^{10}z^2+ 11070wx^{27}y^{12}z^2 \\
&+ 42768wx^{25}y^{14}z^2 + 109368wx^{23}y^{16}z^2 + 177252wx^{21}y^{18}z^2 \\
&+ 196308wx^{19}y^{20}z^2 + 147672wx^{17}y^{22}z^2 + 71424wx^{15}y^{24}z^2\\
&+ 22950wx^{13}y^{26}z^2 + 4830wx^{11}y^{28}z^2 + 744wx^9y^{30}z^2 \\
&+ 18x^{32}y^7z^3 + 204x^{30}y^9z^3+ 1575x^{28}y^{11}z^3 + 7704x^{26}y^{13}z^3 \\
&+ 24234x^{24}y^{15}z^3+ 48456x^{22}y^{17}z^3 + 65310x^{20}y^{19}z^3 \\
&+ 60256x^{18}y^{21}z^3+ 35667x^{16}y^{23}z^3 + 14220x^{14}y^{25}z^3 \\
&+ 3755x^{12}y^{27}z^3 + 744x^{10}y^{29}z^3 + y^{39}z^3. 
\end{align*}
\end{enumerate}
We remark the following: 
\[
J_{\widetilde{Q}_{42},T_1}-J_{\widetilde{Q}_{42},T_2}=x^9 y^9 (x^2 - y^2)^9 (w y - x z)^3. 
\]
\end{ex}

Let $f\in {\rm Harm}_3^{\Aut(\widetilde{Q}_{42})}$. Then 
the harmonic weight enumerators associated with $\widetilde{Q}_{42}$ 
and $f$
%of the extended quadratic residue code of length 42 
are as follows. 
\begin{ex}\label{ex:42-harm}
Let $\widetilde{Q}_{42}$ be the extended binary quadratic residue code of length 42, and 
$f$ be a harmonic function of degree $3$, which is 
an invariant of $\Aut(C)$. 
We remark that $\dim({\rm Harm}_3^{\Aut(\widetilde{Q}_{42})})=1$ and 
$f$ is listed on the homepage of one of the authors \cite{miezaki}. 
Then there exists $c_f$ such that 
\begin{align*}
w_{\widetilde{Q}_{42},f}(x,y)=&c_f(-5740x^{12} y^{12} (x^2 - y^2)^9)\\
=&-5740c_f(x^{30} y^{12} - 9 x^{28} y^{14} + 36 x^{26} y^{16}
- 84 x^{24} y^{18} \\
&+  126 x^{22} y^{20} 
- 126 x^{20} y^{22} + 84 x^{18} y^{24} - 36 x^{16} y^{26} \\
&+ 
 9 x^{14} y^{28} - x^{12} y^{30}). 
\end{align*}
\end{ex}
%\begin{proof}[Proof of Theorem \ref{thm:main1}]
%Let $G:=\Aut(C)$. By {\sc Magma}, 
%$\Harm_3^{G}$ is a one-dimensional space and 
%let $\Harm_3^{G}=\langle f\rangle$. 
%We remark that 
%$f$ is listed on the homepage of one of the authors \cite{miezaki}. 
%By {\sc Magma}, we obtain the results. 
%\end{proof}
We have the following as a corollary. 
\begin{cor}\label{cor:main2}
Let $\widetilde{Q}_{42}$ be the extended binary quadratic residue code of length 42. 
%Then $(\widetilde{Q}_{42})_{10}$ $($also $(\widetilde{Q}_{42})_{32}$$)$ is a $3$-design but not a $4$-design. 
Then $(\widetilde{Q}_{42})_{10}$ $($resp.~$(\widetilde{Q}_{42})_{32}$$)$ is a $3$-$(42,10,18)$ design $($resp.~$3$-$(42, 32, 744)$ design$)$. 
Moreover, $(\widetilde{Q}_{42})_{\ell}$ is not a $3$-design if $\ell\neq 10,32$.%Let $\widetilde{Q}_{42}$ be the extended binary quadratic residue code of length 42. 
%Then $(\widetilde{Q}_{42})_{10}$ $($also $(\widetilde{Q}_{42})_{32}$$)$ is a $3%$-design but not a $4$-design. 
%Moreover, $(\widetilde{Q}_{42})_{\ell}$ is not a $3$-design if $\ell\neq 10,32$. 
%Hence, we have 
%\[
%2=\delta(C)<s(C)=3. 
%\]

%$C_{\ell}$ is a $t$-design with 
%\[
%t=
%\begin{cases}
%3\ (\ell= 10)\\
%2\ (\ell\neq 10)
%\end{cases}. 
%\]
\end{cor}

%In this subsection, 
%we give a proof of Corollary \ref{cor:main}. 
\begin{proof}[Proof of Corollary~\ref{cor:main2}]
First, we show that $(\widetilde{Q}_{42})_{10}$ is a $3$-design. 
We give two proofs. 
%\begin{enumerate}
%\item [(1)]
The first proof relies on the properties of Jacobi polynomials. 
By Example~\ref{ex:42-Jac}, 
%for $j+\ell= 10$, 
the coefficient of $z^{3} x^{32} y^{7}$ in 
$J_{\widetilde{Q}_{42},T_1}-J_{\widetilde{Q}_{42},T_2}$ is zero. 
Hence, 
by Remark~\ref{rem:hom}, 
$(\widetilde{Q}_{42})_{10}$ is a $3$-design. 

%\item [(2)]
The second proof relies on the properties of harmonic weight enumerators. 
By Example~\ref{ex:42-harm}, 
the coefficient of $x^{32} y^{10}$ in 
$w_{\widetilde{Q}_{42},f}$ is zero. 
Hence, 
by Remark~\ref{rem:hom-2}, 
$(\widetilde{Q}_{42})_{10}$ is a $3$-design. 
%\end{enumerate}
%It can be proved similarly that 
%$C_{32}$ is a $3$-design. 

Second, we show that $(\widetilde{Q}_{42})_{10}$ is a $3$-$(42,10,18)$ design. 
Let $(\widetilde{Q}_{42})_{10}$ be a $t$-$(42,10, \lambda_{t})$ design 
with $t \geq 3$. 
By using {\sc Magma}, we have $|(\widetilde{Q}_{42})_{10}|=1722$. 
By Lemma~\ref{lem: divisible}, we have $\lambda_{3}=18$. 
Assuming that $(\widetilde{Q}_{42})_{10}$ is a $4$-design, 
we have 
\[
\lambda_{4}=\frac{42}{13}
\] 
by Lemma~\ref{lem: divisible}, which is a contradiction. 
Hence, $(\widetilde{Q}_{42})_{10}$ is a $3$-$(42,10,18)$ design.

%\end{enumerate}
%It can be proved similarly that 
%$C_{32}$ is a $3$-design. 

Third, we show that for $\ell\neq 10,32$, 
$(\widetilde{Q}_{42})_{\ell}$ is not a $3$-design 
whenever $(\widetilde{Q}_{42})_{\ell}$ is non-empty. 
By Example~\ref{ex:42-harm}, 
the coefficient of $x^{42-\ell} y^{\ell}$ in 
$w_{\widetilde{Q}_{42},f}$ is non-zero. 
Hence, by Theorem~\ref{thm:design-aut}, 
$(\widetilde{Q}_{42})_{\ell}$ is not a $3$-design. 
\end{proof}

\begin{rem}
We introduce the following notations: 
%Let %$C_{w}$ be the support design of a code $C$ for weight $w$ and
\begin{align*}
\delta(C)&:=\max\{t\in \mathbb{N}\mid \forall w, 
C_{w} \mbox{ is a } t\mbox{-design}\};\\ 
s(C)&:=\max\{t\in \mathbb{N}\mid \exists w \mbox{ s.t.~} 
C_{w} \mbox{ is a } t\mbox{-design}\}.
\end{align*}
We remark that $\delta(C) \leq s(C)$ holds, and 
we consider the possible occurrence of $\delta(C)<s(C)$.
In \cite{MN-TEC}, the second and fourth named authors gave 
the first nontrivial examples of a code 
that supports combinatorial $t$-designs for all weights 
obtained from the Assmus--Mattson theorem \cite{assmus-mattson} and 
that supports $t'$-designs for some weights with some $t'>t$ 
(see also \cite{{BS},{Dillion-Schatz},{MMN},{mn-typeI}}). 

%After that, 
%an example was given by Bonnecaze and Sol\'e \cite{BS}, who 
%found a 3-design in 
%the extended quadratic residue code of length 42. 
%They showed the existence of this design by 
%electronic calculation and 
%noted that 
%this design ``cannot be derived from the Assmus--Mattson Theorem, 
%and does not follow by the standard transitivity argument." 
%Therefore, it is natural to ask whether 
%for $\ell\neq 10$, $C_\ell$ is a 3-design. 
%Moreover, can we give a theoretical proof of the fact that 
%$C_{10}$ is a $3$-design? 
%The aim of the present paper is to settle these problems. 

%In this short note, 
%we give Jacobi polynomials and harmonic weight enumerators of 
%the extended quadratic residue code of length 42. 
%Then we present a theoretical proof of the existence of the 3-design in 
%the extended quadratic residue code of length 42. 
%Moreover, we show that 
%$C_{\ell}$ is not a $3$-design if $\ell\neq 10$ . 

%Let $\widetilde{Q}_{42}$ be the extended quadratic residue code of length 42. 
By Corollary~\ref{cor:main2}, we have 
\[
2=\delta(\widetilde{Q}_{42})<s(\widetilde{Q}_{42})=3. 
\]
More examples are given in \cite{Ishikawa}. 
\end{rem}

\section*{Acknowledgments}
The first and second named authors thank Reina Ishikawa 
for helpful discussions on this research. 
%The authors would also like to thank the anonymous reviewers for their
%beneficial comments on an earlier version of the manuscript. 
The second and third named authors were supported 
by Japan Society for Promotion of Science KAKENHI (20K03527, 22K03277).

%%%%%%%%%%%%%%%%%%%%%%%%%%%%%%%%%%%%%%%%%%%%%%%%%%%%%%%%%%%%%%%%%%%%%%%%%%%%%%%%

\end{document}